\documentclass{amsart}

\usepackage[T1]{fontenc}
\usepackage[utf8]{inputenc}
\usepackage{lmodern}
\usepackage{amsmath,amssymb,mathtools}

\usepackage{amsthm} 
\usepackage{graphicx} 
\usepackage{enumerate}
\usepackage{amsfonts} 
\usepackage{ulem} 
\usepackage{float} 
\usepackage{hyperref} 
\usepackage{cite} 
\usepackage[dvipsnames]{xcolor} 
\usepackage{tikz} 
\usepackage{tikz-cd} 
\usepackage{faktor} 
\usepackage{array} 
\usepackage{transparent} 
\usepackage{mathdots} 
\usepackage{adjustbox} 
\usepackage{thmtools} 
\usepackage{thm-restate} 
\usepackage[margin=1.25in]{geometry} 
\usepackage{mathtools} 

\theoremstyle{plain}
\newtheorem{theorem}{Theorem}[section]
\newtheorem{lemma}[theorem]{Lemma}
\newtheorem{proposition}[theorem]{Proposition}
\newtheorem{corollary}[theorem]{Corollary}
\theoremstyle{definition}
\newtheorem{definition}[theorem]{Definition}
\theoremstyle{remark}
\newtheorem{remark}[theorem]{Remark}

\newtheorem*{thm*}{Theorem}

\newtheorem*{ack}{Acknowledgements}

\DeclareMathOperator{\Aut}{Aut}
\DeclareMathOperator{\Out}{Out}
\DeclareMathOperator{\Inn}{Inn}
\DeclareMathOperator{\Comm}{Comm}

\DeclareMathOperator{\GL}{GL}
\DeclareMathOperator{\Homeo}{Homeo}
\DeclareMathOperator{\St}{St}
\DeclareMathOperator{\Rist}{Rist}
\DeclareMathOperator{\stab}{stab}
\DeclareMathOperator{\cyl}{Cyl}
\DeclareMathOperator{\ad}{Ad}

\title[Automorphism-invariant refinements]{Automorphism-invariant refinements of weakly branch actions via overlap functions}
\date{}

\author{Armando Martino}

\begin{document}

\begin{abstract}
Let a finitely generated group $G$ act weakly branch on a locally finite rooted tree $T$ with boundary $\partial T$. The rooted tree structure is encoded by the \textit{overlap function}, which is our name for the Gromov product on the boundary:
\[
  c(\xi,\eta)=|\xi\wedge\eta|.
\]
We axiomatise this function and show that when it is `admissible', one can recover the rooted tree.

By the boundary rigidity theorem of Lavreniuk and Nekrashevych, $\Aut(G)$ acts canonically on $\partial T$. We therefore form the automorphism symmetrisation of the overlap function:
\[
  \widehat c(\xi,\eta)
  =\inf_{\alpha\in\Aut(G)}c(\alpha \cdot\xi,\alpha \cdot\eta).
\]

We prove that \(\widehat c\) is again an admissible \(G\)-invariant overlap
function and that its associated tree \(\widehat T\) is locally finite.
The action of \(G\) on \(\widehat T\) is faithful and weakly branch, and is
branch if and only if the original action on \(T\) is branch. Moreover,
\[
N_{\operatorname{Aut}(\widehat T)}(G)
\cong
\operatorname{Aut}(G).
\]
In particular, \(\operatorname{Aut}(G)\) is weakly branch. We also
describe the finite weakly branch extensions of \(G\): they are precisely
the pullbacks of finite subgroups of \(\operatorname{Out}(G)\). If \(G\)
is branch, all these extensions are branch. In both cases, they act on the same tree
\(\widehat T\).
\end{abstract}

\maketitle

\section{Introduction}

The goal of this paper is to elaborate on the fascinating result of Lavreniuk and Nekrashevych:

\begin{theorem}[{\cite[Theorem 7.3]{LavreniukNekrashevych}}] 
	
	If \( G\) has a weakly branch action on the rooted tree \( T\), then 
	\[
	N_{\Homeo(\partial T)}(G) \cong \Aut(G).
	\]
\end{theorem}

That is, every automorphism of \( G\) induces and is induced by a homeomorphism of the boundary of \( T\). Our result will be to show that, up to changing the underlying tree \( T\) to \( \widehat{T}\), we can realise every automorphism of \( G\) as an automorphism of \( \widehat{T}\), while the action remains weakly branch. Our methods require that \( G\) be finitely generated. As a consequence, \( \Aut(G)\) will also be weakly branch in its action on \( \widehat{T}\).

This belongs to a broader circle of rigidity and reconstruction results
for groups acting on rooted trees and their boundaries; see
\cite{GrigorchukWilsonUniqueness,FarinaAsategui}.

\medskip

Recall that an action of the group \( G\) on a rooted tree \( T\) is called \textit{weakly branch} if \( G\) acts transitively on every level (the vertices at a given distance from the root) and every rigid stabiliser is non-trivial; equivalently, for every vertex \( v\) of \( T\) there is some non-trivial group element \( 1 \neq g_v\) whose support lies entirely in the subtree below \( v\). See Definition~\ref{def:weakly-branch} and \cite{BRS}, \cite{Garrido2015}.

Weakly branch groups form a large and varied class. They include all branch groups---in particular, the classical Grigorchuk and Gupta--Sidki groups and the many families of GGS and spinal groups---but also genuinely non-branch examples such as the Basilica group. More broadly, they arise naturally among self-similar groups and have provided a rich source of finitely generated groups with striking algebraic, geometric and dynamical properties; see, for example, \cite{BRS,Garrido2015}. The first Grigorchuk group was the first known group of intermediate growth and is a finitely generated infinite torsion group. It also provided the first example of an amenable group that is not elementarily amenable \cite{Grigorchukburnside,Grigorchukdegrees}.

\medspace

Our construction is motivated by the use of length functions to encode
group actions on trees. Based length functions and translation length
functions play a central role in geometric group theory, notably in
Chiswell's theorem - see \cite{Chiswell}. They are poorly suited, however, to
actions on rooted trees. Every automorphism of a rooted tree fixes the
root, and hence both its translation length and its displacement of the
root are zero. Choosing a basepoint at level \(n\) gives information only
about a finite truncation of the tree.

The natural geometry instead lies on the boundary. Given two ends
\(\xi,\eta\in\partial T\), define
\[
c(\xi,\eta)=|\xi\wedge\eta|,
\]
where \( |\xi\wedge\eta| \) is the length of their maximal common intial segment, and \( \infty\) if they are equal.  This is precisely the Gromov product
\((\xi\mid\eta)_o\) on the boundary of the rooted tree, although the term
\textit{overlap function} better reflects its elementary tree-theoretic
meaning.

For a fixed end \(\xi\), the quantities
\[
c(\xi,g\xi),
\qquad g\in G,
\]
measure how far along the ray \(\xi\) one travels before the action of
\(g\) moves it elsewhere. This gives a natural replacement for a based
length function. It nevertheless favours a chosen end and its orbit.
We therefore consider the overlap function
\[
c\colon \partial T\times\partial T
\longrightarrow
\mathbb N\cup\{\infty\}.
\]
The entire rooted tree can be reconstructed from this function: the
vertices at level \(n\) are the equivalence classes for
\[
\xi\sim_n\eta
\quad\Longleftrightarrow\quad
c(\xi,\eta)\geq n.
\]
See Proposition~\ref{prop:reconstruction}.

The overlap-function description makes common refinements particularly
transparent. Suppose that \(c_1\) and \(c_2\) are overlap functions on the
same boundary. At level \(n\), the equivalence relation associated to
\[
\min\{c_1,c_2\}
\]
is the intersection of the equivalence relations associated to \(c_1\)
and \(c_2\). Its classes are therefore precisely the nonempty
intersections of a level-\(n\) cylinder for the first tree with a
level-\(n\) cylinder for the second. Consequently,
\(\min\{c_1,c_2\}\) encodes the common refinement of the two rooted trees,
simultaneously at every level.

By the theorem of Lavreniuk
and Nekrashevych above, when \( G\) has a weakly branch action on \( T\),  \(\Aut(G)\) acts canonically and faithfully on
\(\partial T\), and this action satisfies
\[
\alpha\cdot(g\xi)
=
\alpha(g)(\alpha\cdot\xi)
\]
for every \(\alpha\in\Aut(G)\), \(g\in G\), and
\(\xi\in\partial T\).

We define the twisted overlap function by
\[
c^\alpha(\xi,\eta)
=
c\bigl(\alpha \cdot\xi,\alpha \cdot \eta\bigr).
\]
Alternatively, one may pre-compose the original action of \(G\) on \(T\)
by \(\alpha\), obtaining the twisted \(G\)-tree \(\alpha T\). We show
that these two operations agree: the tree \(T_{c^\alpha}\) reconstructed
from \(c^\alpha\) is naturally \(G\)-equivariantly isomorphic to
\(\alpha T\).

The central construction of the paper is the common refinement of all
these twisted trees. In terms of overlap functions, it is given by the
single formula
\[
\widehat c
=
\inf_{\alpha\in\operatorname{Aut}(G)}c^\alpha.
\]
By construction, \(\widehat c\) is invariant under the canonical action
of \(\operatorname{Aut}(G)\) on the boundary. The principal difficulty is
that an infinite common refinement of finite clopen partitions need not
itself have finitely many parts. Thus it is not immediate that
\(\widehat c\) determines a locally finite tree.

The key input is Lemma~\ref{lem:finiteness}. It states that if a group \( G\) acts by homeomorphisms on a topological space and admits a dense orbit, then the number of \(G \)-invariant partitions into \( d\) clopen subsets is bounded by the number of subgroups of index \( d\). In particular, only finitely many such partitions exist when \(G\) is finitely generated. This applies to the boundary action of any finitely generated group acting level-transitively on a rooted tree, since every orbit on the boundary is then dense.

%

At level \(n\), every
twisted function \(c^\alpha\) gives a \(G\)-invariant clopen partition of
\(\partial T\) with
\[
k_n=|T_n|
\]
parts. Only finitely many distinct such partitions can therefore occur.
Although \(\widehat c\) is defined as an infimum over all automorphisms,
at each fixed level it is consequently determined by a finite common
refinement.

This gives our main result, Theorem~\ref{thm:main}, with the following consequences.
Let \(G\) be a finitely generated group acting faithfully and weakly branch
on a locally finite rooted tree \(T\). We construct another locally finite
rooted tree \(\widehat{T}\) such that:
\begin{itemize}
	\item \(G\) acts faithfully and weakly branch on \(\widehat{T}\);
	\item there is a natural level-preserving surjective simplicial
	\(G\)-map
	\[
	p\colon \widehat{T}\longrightarrow T;
	\]
	\item the action of \(G\) on \(\widehat{T}\) is branch if and only if
	its action on \(T\) is branch;
	\item \(\operatorname{Aut}(G)\) acts faithfully and weakly branch on
	\(\widehat{T}\), and its restriction to
	\(\operatorname{Inn}(G)\cong G\) is the \(G\)-action;
	\item conjugation induces an isomorphism
	\[
	N_{\operatorname{Aut}(\widehat{T})}(G)
	\cong \operatorname{Aut}(G).
	\]
\end{itemize}

We note that \cite[Theorem 7.5]{LavreniukNekrashevych} proves this for \textit{saturated} actions. In that sense, our contribution is to remove the need for that hypothesis at the cost of changing the tree. In fact, in the case of saturated actions our construction would yield that \( \widehat{T} = T\). We also construct toy examples showing that \( T \neq \widehat{T}\) in general, in Section~\ref{sec:example}.

The construction also has consequences for finite extensions, Theorem~\ref{thm:finite-extensions}. Suppose that \(G\)
is finitely generated and weakly branch, and \( \Gamma \) contains \( G\) as a finite index normal subgroup. Then \( \Gamma \) is weakly branch if and only if the natural conjugation map is injective: 

\[
\Gamma \hookrightarrow \Aut(G).
\]

Moreover, when this map is injective, \( \Gamma\) is branch if and only if \( G\) is branch.  
Hence, finite extensions of \( G\) that are (weakly) branch are all realised as pullbacks of finite subgroups of \( \Out(G)\).  Moreover, all of them act as (weakly) branch groups on the same tree
\(\widehat T\).

Finally, we show how one might apply these techniques for the specific case of Grigorchuk's group in Section~\ref{sec:grigorchuk-extensions}, although this is mostly for interest and much of the material there is already known.

\begin{ack}
	I would like to thank Jorge Fari\~{n}a-Asategui for an inspiring talk on his work from \cite{FarinaAsategui}, which motivated me to think about this problem. I would also like to thank Laurent Bartholdi for encouraging comments on an earlier draft of this paper. 
\end{ack}

\section{Rooted trees and branch actions}

Let $T$ be a rooted tree with root $o$. For $n\geq 0$, write
\[
  T_n=\{v\in T:d(o,v)=n\}
\]
for the $n$-th level of $T$. If $v\in T$, the rooted subtree consisting of $v$ and all its descendants is denoted by $T_v$.

We shall assume that rooted trees are locally finite and have no leaves. The boundary $\partial T$ is the set of (geodesic) infinite rays beginning at $o$. Here we consider a ray to be a non-backtracking infinite edge-path or vertex-path.

This boundary has a natural topology whose basis is given by the cylinders: given a vertex $v$ of $T$, the cylinder at $v$ consists of all the (geodesic) infinite rays starting at $o$ which pass through $v$. 

We shall denote the cylinder at the vertex \( v \) by \( \cyl(v) \).

Cylinder sets are clopen. If $T$ is locally finite, this topology makes $\partial T$ a compact, metrisable and totally disconnected space. In particular, $\partial T$ is Hausdorff.

If, in addition, $\partial T$ has no isolated points, for example, if every ray encounters branching infinitely often, then $\partial T$ is a Cantor space.

The group $\Aut(T)$ consists of the graph automorphisms of $T$ which fix the root. Every element of $\Aut(T)$ preserves levels and induces a homeomorphism of $\partial T$.

Let $G\leq\Aut(T)$. For a vertex $v\in T$, its stabiliser is
\[
  \St_G(v)=\{g\in G:gv=v\}.
\]
The rigid stabiliser of $v$ is
\[
  \Rist_G(v)=\{g\in G:gw=w\text{ for every }w\notin T_v\}.
\]
Thus an element of $\Rist_G(v)$ acts trivially outside the subtree rooted at $v$.

For $n\geq 0$, the pointwise stabiliser of the $n$-th level is
\[
  \St_G(n)=\bigcap_{v\in T_n}\St_G(v),
\]
and the rigid stabiliser of the $n$-th level is
\[
  \Rist_G(n)=\langle\Rist_G(v):v\in T_n\rangle.
\]
Since the subtrees $T_v$, $v\in T_n$, are pairwise disjoint, the groups $\Rist_G(v)$ commute pairwise, and hence
\[
  \Rist_G(n)=\prod_{v\in T_n}\Rist_G(v).
\]

\begin{definition}
The action $G\curvearrowright T$ is \textit{level-transitive} if $G$ acts transitively on $T_n$ for every $n\geq 0$.
\end{definition}

\begin{definition}\label{def:weakly-branch}
A group $G\leq\Aut(T)$ is \textit{weakly branch} on $T$ if its action is level-transitive and
\[
  \Rist_G(v)\neq 1
\]
for every vertex $v\in T$.

Equivalently, $G$ is weakly branch if it is level-transitive and $\Rist_G(v)$ is infinite for every $v$.

For a level-transitive action, it is enough to require that one rigid stabiliser at each level be nontrivial.
\end{definition}

\begin{definition}
A group $G\leq\Aut(T)$ is \textit{branch} on $T$ if its action is level-transitive and
\[
  [G:\Rist_G(n)]<\infty
\]
for every $n\geq 0$.
\end{definition}

Every infinite branch group is weakly branch. Indeed, if $\Rist_G(n)$ has finite index, then it is nontrivial; by level-transitivity, all vertex rigid stabilisers at level $n$ are conjugate, and hence each of them is nontrivial.

\begin{remark}
	Every \( g \in G \leq \Aut(T)\) induces a homeomorphism of the boundary, since cylinders are mapped to cylinders. Moreover, if \( \cyl(v)\) is the cylinder at \( v\), then \( g \in \Rist_G(v) \) if and only if its action is supported in \( \cyl(v)\).  
\end{remark}

We note some elementary results.

\begin{lemma}
	\label{lem:centreless}
	Let \( G\) be weakly branch. Then \( G\) is centreless. Thus, \( G \cong \Inn(G)\).
\end{lemma}
\begin{proof}
	If \( z \in Z(G)\) then, for any vertex, 
	\[
	\Rist_G(v) = z \Rist_G(v) z^{-1} = \Rist_G(zv).
	\]
	Since rigid stabilisers are all non-trivial, and rigid stabilisers of distinct vertices at the same level intersect trivially, we deduce that \( z\) acts trivially on the tree and hence \( z =1\), as the action is faithful.
\end{proof}

\begin{lemma}
	\label{lem:no finite normal subs}
	Let \( G\) be weakly branch. Then any finite normal subgroup of \( G\) is trivial. 
\end{lemma}
\begin{proof}
	Let \(N\) be a finite normal subgroup of \(G\). Since the action is
	faithful, there is a level \(n\) such that
	\[
	N\cap\St_G(n)=1.
	\]
	Suppose that \(1\neq x\in N\). Then \(x\) moves some vertex \(v\) at
	level \(n\). Choose
	\[
	1\neq r\in\Rist_G(v).
	\]
	Since \(N\) is normal,
	\[
	[x,r]\in N.
	\]
	Both \(r\) and \(xrx^{-1}\) fix level \(n\) pointwise, so
	\[
	[x,r]\in\St_G(n).
	\]
	It follows that
	\[
	[x,r]\in N\cap\St_G(n)=1.
	\]
	Thus \(xrx^{-1}=r\). But
	\[
	xrx^{-1}\in\Rist_G(xv),
	\]
	where \(xv\neq v\), again contradicting the fact that rigid stabilisers
	of distinct vertices on the same level have trivial intersection.
	Therefore \(N=1\).
\end{proof}

\section{Overlap functions and rooted trees}

This section is concerned with encoding a locally finite rooted tree \( T\) via an `overlap function'. The idea here is that vertices of \( T\) give rise to cylinders, which are clopen subsets of \( \partial T\). Vertices at level \( n\) then give a finite partition of  \( \partial T\) into clopen subsets. 

Therefore, we can view the tree \( T\) as being (equivalent to) a nested sequence of finite clopen partitions of \( \partial T\). An admissible overlap function - Definitions~\ref{def:overlap} and \ref{def: admissible} - is just a single object for recording that data. 

\medspace

Throughout this paper, the natural numbers \( \mathbb{N}\) shall be assumed to contain \( 0\).

\begin{definition}
	\label{def:overlap}
Let $X$ be a set. An \textit{overlap function} on $X$ is a function
\[
  c:X\times X\longrightarrow\mathbb N\cup\{\infty\}
\]
such that, for all $\xi,\eta,\zeta\in X$,
\begin{enumerate}[(i)]
\item $c(\xi,\xi)=\infty$;
\item $c(\xi,\eta)=c(\eta,\xi)$;
\item
\(
  c(\xi,\zeta)\geq\min\{c(\xi,\eta),c(\eta,\zeta)\}.
\)
\end{enumerate}
\end{definition}
\begin{remark}
	Note that this is just the Gromov product where \( X\) will be the boundary of a tree. Condition (iii) is then 0-hyperbolicity.
\end{remark}

We shall use overlap functions to re-construct trees. This is a well-understood idea. Indeed, with the convention \(2^{-\infty}=0\), the function
\[
d_c(\xi,\eta)=2^{-c(\xi,\eta)}
\]
is the natural (pseudo) ultrametric on the boundary. The correspondence between
rooted trees and ultrametric end spaces is developed systematically
in~\cite{Hughes}.

\begin{definition}
	\label{def:level n equivalence}
	Given an overlap function \( c\) on a set \( X\) we define for each $n\geq 0$ an equivalence relation
	\[
	\xi\sim_n^c\eta\quad\Longleftrightarrow\quad c(\xi,\eta)\geq n.
	\]
	This is the level \( n\) partition of \(X \) via \( c\).
\end{definition}
\begin{remark}
	The third overlap function axiom is precisely what is needed for transitivity.
\end{remark}

\begin{definition}
	\label{def: admissible}
Suppose that $X$ is compact and totally disconnected. An overlap function $c$ on \( X\) is called \textit{{admissible}} if:
\begin{enumerate}[(i)]
\item every relation $\sim_n^c$ has finitely many equivalence classes, each of which is clopen;
\item \( c\) separates points:
\(
  c(\xi,\eta)=\infty\quad\Longrightarrow\quad\xi=\eta.
\)
\end{enumerate}
\end{definition}

\begin{remark}
Note that a space which admits an admissible overlap function must be totally separated and hence totally disconnected.
\end{remark}

\begin{proposition}[Reconstruction from an overlap function]\label{prop:reconstruction}
Let $c$ be an admissible overlap function on a compact and totally disconnected space $X$. There is a locally finite rooted tree $T_c$ whose vertices at level $n$ are the classes of $\sim_n^c$, with adjacency given by containment of consecutive classes. The map
\[
  \pi_c: X\longrightarrow\partial T_c,
  \qquad
  \xi  \stackrel{\pi_c}{\longmapsto} \bigl([\xi]_{\sim_n^c}\bigr)_{n\geq 0},
\]
is a homeomorphism.

If a group $G$ acts on $X$ and
\[
  c(g\xi,g\eta)=c(\xi,\eta)
\]
for all $g\in G$, then $G$ acts by rooted automorphisms on $T_c$, and the homeomorphism above is \( G\)-equivariant.
\end{proposition}

\begin{proof}
The vertices of $T_c$ are the equivalence classes of $\sim_n^c$, with $n$ giving the distance from the root.

The relations are nested:
\[
  \sim_{n+1}^c\ \subseteq\ \sim_n^c.
\]
Hence every class at level $n+1$ lies in a unique class at level $n$, which defines the edges of $T_c$. Finiteness of the classes at each level gives local finiteness.

The displayed map is injective because the relations separate points. It is surjective because a ray in $T_c$ gives a nested sequence of nonempty compact classes. Their intersection is nonempty by compactness and consists of exactly one point by separation.

The map is continuous since the pre-image of a cylinder is precisely the equivalence class of some $\sim_n^c$. The map is then a continuous bijection from a compact space to a Hausdorff space and hence a homeomorphism.

If $c$ is $G$-invariant, then $G$ preserves every relation $\sim_n^c$ and therefore acts on the corresponding rooted tree. Moreover, the homeomorphism above is clearly \( G\)-equivariant in this case since \( ([g\xi]_{\sim_n^c}) = (g[\xi]_{\sim_n^c})  \). 
\end{proof}

\begin{corollary}
	\label{cor: injective to boundary}
	Let \( X\) be a compact, totally disconnected space with admissible overlap function, \( c\). Then the tree \( T_c\) constructed via Proposition~\ref{prop:reconstruction} has no leaves. In particular the map, 
	\[
	\Aut(T_c) \to \Homeo(\partial T_c), 
	\]
	is injective. 
\end{corollary}
\begin{proof}
	Every vertex \([\xi]_{\sim_n^c}\) has a child \([\xi]_{\sim_{n+1}^c}\), so \(T_c\) has no leaves. The stated map is now clearly injective, since every vertex lies on some infinite ray from the root. 
\end{proof}


\begin{corollary}
	\label{cor:T=T_c}
	Let \(G\leq \Aut(T)\), where \(T\) is a locally finite rooted tree without leaves. Then
	\[
	c\colon \partial T\times\partial T\longrightarrow \mathbb N\cup\{\infty\},
	\qquad
	c(\xi,\eta):=|\xi\wedge\eta|,
	\]
	is an admissible \(G\)-invariant overlap function.
	
	For each \(n\), the map
	\[
	T_n\longrightarrow \partial T/{\sim_n^c},
	\qquad
	v\longmapsto \cyl(v),
	\]
	is a \(G\)-equivariant bijection, mapping a vertex to the cylinder it defines. Consequently, these maps assemble to a
	\(G\)-equivariant isomorphism
	\[
	T\longrightarrow T_c.
	\]
\end{corollary}

\begin{proof}
	The overlap axioms and \(G\)-invariance are immediate. For each \(n\), two rays
	are \(\sim_n^c\)-equivalent precisely when they pass through the same vertex at
	level \(n\). Thus the \(\sim_n^c\)-classes are exactly the level-\(n\)
	cylinders. They are finite in number and clopen, and the relations separate
	points, so \(c\) is admissible.
	
	Since \(T\) has no leaves, every vertex determines a nonempty cylinder, and
	hence the displayed map is a bijection on each level. Adjacency in both trees
	is given by containment of cylinders at consecutive levels. Finally,
	\[
	g\cyl(v)=\cyl(gv),
	\]
	so the resulting isomorphism is \(G\)-equivariant.
\end{proof}

%


\begin{lemma}\label{lem:overlap leq}
	Let \(X\) be a compact totally disconnected space, and let
	\[
	c,d\colon X\times X\longrightarrow\mathbb{N}\cup\{\infty\}
	\]
	be admissible overlap functions satisfying
	\[
	c\leq d.
	\]
	
	Then the following hold.
	\begin{enumerate}[(i)]
		\item There is a natural level-preserving surjective simplicial map
		\[
		p_{c,d}\colon T_c\longrightarrow T_d
		\]
		given on the vertices at level \(n\) by
		\[
		p_{c,d}\bigl([\xi]_{\sim_n^c}\bigr)
		=
		[\xi]_{\sim_n^d}.
		\]
		
		\item 
		
		 We have a commuting diagram,

		\[
		\begin{tikzcd}
			& X \arrow[dl, "\pi_c"'] \arrow[dr, "\pi_d"] & \\
			\partial T_c \arrow[rr, "\partial p_{c,d}"'] & & \partial T_d
		\end{tikzcd}
		\]
		where \( \partial p_{c,d}\) is the map induced on the boundaries by \( p_{c,d}\). 
		
		\item If a group \(G\) acts on \(X\), and both \(c\) and \(d\) are
		\(G\)-invariant, then \(p_{c,d}\) and \( \partial p_{c,d}\) are \(G\)-equivariant.
	\end{enumerate}
\end{lemma}

\begin{proof}
Recall the definition of \( \sim_n^c \) from Definition~\ref{def:level n equivalence}.

	Since \(c\leq d\), we have
	\[
	\xi \sim_n^c\eta
	\quad\Longrightarrow\quad
	\xi {\sim_n^d} \eta.
	\]

	Thus every \( \sim_n^c\)-class is contained in a unique
	\(\  {\sim_n^d}\)-class. We may therefore define
	\[
	p_{c,d}\bigl([\xi]_{\sim_n^c}\bigr)
	=
	[\xi]_{\sim_n^d}.
	\]
	This is well defined, is surjective and preserves levels and edges.  Hence \(p_{c,d}\) is a level-preserving surjective simplicial map, establishing (i).

	The map \( \partial p_{c,d}\) simply maps \( \pi_c( \xi) \) to \( \pi_d(\xi)\), hence we get a commuting diagram as in (ii).

	Finally, suppose that \(G\) acts on \(X\) and that both \(c\) and \(d\)
	are \(G\)-invariant. For \(g\in G\), we have
	\[
	\begin{aligned}
		p_{c,d}\bigl(g[\xi]_{\sim_n^c}\bigr)
		&=
		p_{c,d}\bigl([g\xi]_{\sim_n^c}\bigr)\\
		&=
		[g\xi]_{\sim_n^d}\\
		&=
		g[\xi]_{\sim_n^d}\\
		&=
		g\,p_{c,d}\bigl([\xi]_{\sim_n^c}\bigr).
	\end{aligned}
	\]
	Thus \(p_{c,d}\) is \(G\)-equivariant. Similarly, \( \partial p_{c,d}\) will also be equivariant by Proposition~\ref{prop:reconstruction}. 
\end{proof}
%
%

\section{Boundary rigidity}

For the remainder of the paper, let
\[
G\leq\Aut(T)
\]
be a finitely generated group acting faithfully and weakly branch on a
locally finite rooted tree \(T\).

We use the following boundary rigidity theorem of Lavreniuk and
Nekrashevych.

\begin{theorem}[{\cite[Theorem~7.3 and Lemma~5.4]
		{LavreniukNekrashevych}}]\label{thm:boundary-rigidity}
	Let \(G\leq\Aut(T)\) act faithfully and weakly branch on a locally finite
	rooted tree \(T\). Then conjugation induces an isomorphism
	\[
	N_{\Homeo(\partial T)}(G)\cong\Aut(G).
	\]
	Consequently, \(\Aut(G)\) acts canonically and faithfully on
	\(\partial T\), and this action satisfies
	\[
	\alpha\cdot(g\xi)
	=
	\alpha(g)(\alpha\cdot\xi)
	\]
	for every \(\alpha\in\Aut(G)\), \(g\in G\), and
	\(\xi\in\partial T\).
\end{theorem}

\begin{proof}
	Lavreniuk and Nekrashevych prove that every automorphism of \(G\) is
	induced by a homeomorphism of \(\partial T\), and that the centraliser of
	\(G\) in \(\Homeo(\partial T)\) is trivial. Thus conjugation gives a
	surjective homomorphism
	\[
	N_{\Homeo(\partial T)}(G)\longrightarrow\Aut(G)
	\]
	with trivial kernel.
\end{proof}

\begin{remark}
	Let \(c\) and \(d\) be admissible \(G\)-invariant overlap functions on
	\(\partial T\). Then \(T_c\) and \(T_d\) are \(G\)-equivariantly isomorphic
	if and only if \(c=d\).
	
	Indeed, a \(G\)-equivariant isomorphism \(T_c\to T_d\) induces a
	\(G\)-equivariant homeomorphism of their boundaries. Under the natural
	identifications of both boundaries with \(\partial T\), this homeomorphism
	centralises \(G\), and is therefore the identity by
	Theorem~\ref{thm:boundary-rigidity}. The isomorphism preserves lengths of
	common initial segments, so \(c=d\). The converse is immediate.
\end{remark}

The following finiteness lemma will imply that the symmetrised overlap function is admissible; it is really the key technical lemma of this paper.

%
%
%
%

\begin{lemma}\label{lem:finiteness}
	Let a group $G$ act by homeomorphisms on a topological space $X$.
	Suppose that the action has a dense orbit.
	
	For every $d\geq 1$, the number of \(G\)-partitions of \(X\) into exactly
	\(d\) clopen subsets is bounded above by the number of subgroups of index
	\(d\) in \(G\).
	
	In particular, if \(G\) has only finitely many subgroups of index \(d\),
	then there are only finitely many such partitions. This holds, in
	particular, when \(G\) is finitely generated.
\end{lemma}

\begin{proof}
Suppose that the orbit of \( x \in X\) is dense. Now consider a \(G\)-partition into \( d\) clopen subsets. Let \( C\) be the part of the partition containing \( x\) and \( H = \stab_G(C)\). We claim that \( H\) has index \( d\)	and the partition is determined by \( H\). 
	
First note that since the orbit of \( x\) is dense, the \(G\) orbit of \(C\) must cover \( X\); hence the induced \( G\) action on the partition is transitive and hence \( H\) has index \( d\). The closure \( \overline{Hx}\) of \( Hx \) is clearly contained in \( C\), as \( C\) is closed. The complement of \( \overline{Hx}\) in \( C\) is open as \( C\) is open and so, if non-empty, contains some \( gx\) as \( Gx\) is dense. Since the partition is a \( G\)-partition, we get \( g \in H\) and therefore \( \overline{Hx} = C\). The other parts of the partition are given by \( gC =  \overline{gHx}\).

Therefore, the entire partition is determined by the index \(d\)
subgroup \(H\), namely the stabiliser of the part containing \(x\).
Thus the map which assigns to a \(G\)-partition the stabiliser of its
part containing \(x\) is injective. This proves the claimed bound.

If \(G\) has only finitely many subgroups of index \(d\), it follows
that there are only finitely many such partitions. In particular, this
holds when \(G\) is finitely generated.

%

%
%
%
%
\end{proof}

\begin{remark}
The boundary action of a level-transitive group acting on a rooted tree is clearly minimal; \textit{every} orbit is dense, since the orbit of any point on the boundary meets every cylinder. 
\end{remark}


\section{The automorphism symmetrisation}

In this section we show that the automorphism symmetrisation of an admissible overlap function is again an admissible overlap function. We begin by analysing the effect of twisting by a single automorphism.

\begin{proposition}\label{prop:twisted-overlap}
	Let \(G\leq\Aut(T)\) act faithfully and weakly branch on a locally finite
	rooted tree \(T\), and let
	\[
	c\colon\partial T\times\partial T
	\longrightarrow \mathbb{N}\cup\{\infty\}
	\]
	be the overlap function associated to \(T\).
	
	Let \(\alpha\in\Aut(G)\), acting on \(\partial T\) via the canonical
	boundary action of Theorem~\ref{thm:boundary-rigidity}. 
	
	
	Define
	\[
	c^\alpha(\xi,\eta)
	=
	c(\alpha\cdot\xi,\alpha\cdot\eta)
	\]
	for \(\xi,\eta\in\partial T\). Also, let \(\alpha T\) denote the rooted
	tree \(T\) equipped with the precomposed \(G\)-action
	\[
	g\cdot_{\alpha T}v
	=
	\alpha(g)\cdot_Tv.
	\]
	
	Then the following hold.
	\begin{enumerate}[(i)]
		\item The function \(c^\alpha\) is a \(G\)-invariant admissible overlap function on
		\(\partial T\).
		
		
		\item There is a natural \(G\)-equivariant isomorphism of rooted trees
		\[
		\Phi_\alpha\colon T_{c^\alpha}\longrightarrow \alpha T.
		\]
		
%
	\end{enumerate}
\end{proposition}

\begin{proof}
	Note that \( c^{\alpha}\) is the pullback of \(c\) via the boundary homeomorphism given by \( \alpha\), so is an admissible overlap function.

	To prove \(G\)-invariance, let \(g\in G\) and
	\(\xi,\eta\in\partial T\). Using the compatibility of the canonical
	boundary action with the automorphism \(\alpha\), we obtain
	\[
	\begin{aligned}
		c^\alpha(g\xi,g\eta)
		&=
		c\bigl(\alpha\cdot(g\xi),\alpha\cdot(g\eta)\bigr)\\
		&=
		c\bigl(\alpha(g)(\alpha\cdot\xi),
		\alpha(g)(\alpha\cdot\eta)\bigr)\\
		&=
		c(\alpha\cdot\xi,\alpha\cdot\eta)\\
		&=
		c^\alpha(\xi,\eta),
	\end{aligned}
	\]
	where the third equality follows from the \(G\)-invariance of \(c\).
	
	Recall that the vertices of \(T_{c^\alpha}\) at level \(n\) are the
	\(\sim_n^{c^{\alpha}}\)-classes. Define
	\[
	\Phi_\alpha\bigl([\xi]_{\sim_n^{c^{\alpha}}}\bigr)
	=
	[\alpha\cdot\xi]_{\sim_n^{c^{}}}.
	\]
	The definition of \(c^\alpha\) shows that this map is well defined and injective
	on every level. Since the action of \(\alpha\) on \(\partial T\) is a
	bijection, it is also surjective on every level. Moreover, it preserves
	levels and the containment of cylinder sets, and hence preserves the
	parent relation. Therefore \(\Phi_\alpha\) is an isomorphism of rooted
	trees.
	
	Finally, for \(g\in G\),
	\[
	\begin{aligned}
		\Phi_\alpha\bigl(g[\xi]_{\sim_n^{c^{\alpha}}}\bigr)
		&=
		\Phi_\alpha\bigl([g\xi]_{\sim_n^{c^{\alpha}}}\bigr)\\
		&=
		[\alpha\cdot(g\xi)]_{\sim_n^{c}}\\
		&=
		[\alpha(g)(\alpha\cdot\xi)]_{\sim_n^{c}}\\
		&=
		g\cdot_{\alpha T}[\alpha\cdot\xi]_{\sim_n^{c}}\\
		&=
		g\cdot_{\alpha T}
		\Phi_\alpha\bigl([\xi]_{\sim_n^{c^{\alpha}}}\bigr).
	\end{aligned}
	\]
	Thus \(\Phi_\alpha\) is \(G\)-equivariant.
\end{proof}

We are now ready to fully symmetrise our admissible overlap functions. This is our main theorem.

\begin{theorem}\label{thm:main}
	Let
	\[
	G\leq\Aut(T)
	\]
	be a finitely generated group acting weakly branch on a locally finite
	rooted tree \(T\), and let
	\[
	c\colon\partial T\times\partial T
	\longrightarrow \mathbb{N}\cup\{\infty\}
	\]
	be the overlap function associated to \(T\).
	
	For each \(\alpha\in\Aut(G)\), define
	\[
	c^\alpha(\xi,\eta)
	=
	c(\alpha\cdot\xi,\alpha\cdot\eta),
	\]
	where \(\Aut(G)\) acts on \(\partial T\) via the canonical boundary action
	of Theorem~\ref{thm:boundary-rigidity}. Define
	\[
	\widehat c(\xi,\eta)
	=
	\inf_{\alpha\in\Aut(G)}c^\alpha(\xi,\eta).
	\]
	
	Then the following hold.
	\begin{enumerate}[(i)]
		\item The function \(\widehat c\) is an admissible \(G\)-invariant
		overlap function on \(\partial T\). Moreover, it is invariant under
		the canonical action of \(\Aut(G)\) on \(\partial T\).
		
		Consequently, \(\widehat c\) determines a locally finite rooted tree,
		which we denote by
		\[
		\widehat T=T_{\widehat c}.
		\]
		
		\item There is a natural level-preserving surjective simplicial
		\(G\)-map
		\[
		p\colon\widehat T\longrightarrow T.
		\]
		Under the natural identification (given by Proposition~\ref{prop:reconstruction})
		\[
		\partial\widehat T\cong\partial T,
		\]
		the induced map on boundaries is the identity.
		
		\item The action of \(G\) on \(\widehat T\) is faithful and weakly
		branch. Moreover, the action of \(G\) on \(\widehat T\) is branch if
		and only if the action of \(G\) on \(T\) is branch.
		
		\item The canonical action of \(\Aut(G)\) on \(\partial T\) is induced
		by a faithful action of \(\Aut(G)\) on \(\widehat T\) by rooted tree
		automorphisms, and conjugation induces an isomorphism
		\[
		N_{\Aut(\widehat T)}(G)\cong\Aut(G).
		\]
		
		\item The group \(\Aut(G)\) is weakly branch in its action on
		\(\widehat T\). The action of \( \Inn(G)\) is simply the action of \( G\) under the isomorphism, \( G \cong \Inn(G)\).
	\end{enumerate}
\end{theorem}
\begin{proof}
	We first argue that \( \widehat{c}\) is an overlap function. Note that conditions (i) and (ii) of Definition~\ref{def:overlap} are clear. For condition (iii), we note that for  every $\alpha$,
	\[
	c^{\alpha}(\xi,\zeta)\geq\min\{c^{\alpha}(\xi,\eta),c^{\alpha}(\eta,\zeta)\},
	\]
	since each \( c^{\alpha}\) is an overlap function. 
	Taking infima gives
	\begin{align*}
		\widehat{c}(\xi,\zeta)
		&\geq \inf_{\alpha}\min\{c^{\alpha}(\xi,\eta),c^{\alpha}(\eta,\zeta)\}\\
		&\geq \min\left\{\inf_{\alpha} c^{\alpha}(\xi,\eta),\inf_{\alpha} c^{\alpha}(\eta,\zeta)\right\}\\
		&=\min\{\widehat{c}(\xi,\eta),\widehat{c}(\eta,\zeta)\}.
	\end{align*}

	Hence \( \widehat{c} \) is an overlap function.

	For \(n\geq 0\) and \(\alpha\in\Aut(G)\), let
	\[
	\mathcal{P}_n^\alpha
	\]
	denote the partition of \(\partial T\) into the equivalence classes of
	the relation \( \sim_n^{c^{\alpha}}\). 
	By Proposition~\ref{prop:twisted-overlap}, each \(c^\alpha\) is an
	admissible \(G\)-invariant overlap function, and
	\[
	T_{c^\alpha}\cong_G\alpha T.
	\]
	In particular, if \(T\) has \(k_n\) vertices at level \(n\), then every
	partition \(\mathcal{P}_n^\alpha\) has exactly \(k_n\) parts.
	
	By Lemma~\ref{lem:finiteness}, for each fixed \(n\), only finitely many
	distinct partitions \(\mathcal{P}_n^\alpha\) occur as
	\(\alpha\) ranges over \(\Aut(G)\). Moreover,
	\[
	\begin{aligned}
		\widehat c(\xi,\eta)\geq n
		&\quad\Longleftrightarrow\quad
		c^\alpha(\xi,\eta)\geq n
		\text{ for every }\alpha\in\Aut(G).
	\end{aligned}
	\]
	Thus the equivalence relation associated to \(\widehat c\) at level
	\(n\) is the intersection of finitely many equivalence relations
	associated to admissible overlap functions. Its classes are therefore
	the nonempty intersections of finitely many clopen sets, and hence form
	a finite clopen partition of \(\partial T\).


	Since the identity automorphism belongs to
	\(\Aut(G)\), we have
	\[
	\widehat c\leq c.
	\]
	Consequently, if \(\xi\neq\eta\), then
	\[
	\widehat c(\xi,\eta)\leq c(\xi,\eta)<\infty.
	\]
	It follows that \(\widehat c\) is an admissible overlap function.
	
	Each \(c^\alpha\) is \(G\)-invariant, so \(\widehat c\) is also
	\(G\)-invariant. Furthermore, for every \(\beta\in\Aut(G)\),
	\[
	\begin{aligned}
		\widehat c(\beta\cdot\xi,\beta\cdot\eta)
		&=
		\inf_{\alpha\in\Aut(G)}
		c\bigl(\alpha\cdot(\beta\cdot\xi),
		\alpha\cdot(\beta\cdot\eta)\bigr)\\
		&=
		\inf_{\alpha\in\Aut(G)}
		c\bigl((\alpha\beta)\cdot\xi,
		(\alpha\beta)\cdot\eta\bigr)\\
		&=
		\widehat c(\xi,\eta),
	\end{aligned}
	\]
	since right multiplication by \(\beta\) permutes \(\Aut(G)\). This proves
	the first assertion and allows us to define
	\[
	\widehat T=T_{\widehat c}.
	\]
	
	Since \(\widehat c\leq c\), Lemma~\ref{lem:overlap leq} and Corollary~\ref{cor:T=T_c} give a natural
	level-preserving surjective simplicial map
	\[
	p\colon\widehat T\longrightarrow T.
	\]
	Both \(\widehat c\) and \(c\) are \(G\)-invariant, so \(p\) is
	\(G\)-equivariant. 
	
	By Lemma~\ref{lem:overlap leq}, \( \partial p\) is simply \( \pi_{\widehat{c}}^{-1}\). Or, in other words, 
 under the natural identification
	\[
	\partial\widehat T\cong\partial T,
	\]
	the induced boundary map is the identity. This proves the second assertion.
	
	For the third assertion, we note that the existence of the map \( p\) shows that the action of \( G\) on \( \widehat{T}\) is faithful, since the action on \( T\) is faithful.

	We next show that the action of \(G\) on \(\widehat T\) is weakly
	branch. Since the action of \(G\) on \(T\) is level-transitive, its
	action on \(\partial T\) is minimal: every \(G\)-orbit meets every
	cylinder and is therefore dense. Each level of \(\widehat T\) is a
	finite \(G\)-invariant clopen partition of \(\partial T\). Minimality
	therefore implies that \(G\) acts transitively on every level of
	\(\widehat T\).
	
	Let \(\widehat v\) be a vertex of \(\widehat T\), and let
	\(\cyl({\widehat v}) \subseteq\partial T\) be the corresponding cylinder.
	Since \(\cyl({\widehat v})\) is open and the cylinders arising from \(T\)
	form a basis for the topology on \(\partial T\), there is a vertex
	\(w\) of \(T\) such that
	\[
	\cyl(w) \subseteq \cyl({\widehat v}) .
	\]
	As the action of \(G\) on \(T\) is weakly branch,
	\(\Rist_G^T(w)\neq 1\). Every element of \(\Rist_G^T(w)\) induces a homeomorphism of \( \partial T\) that is supported
	inside \( \cyl(w)\), and hence inside \( \cyl({\widehat v}) \). Therefore
	\[
	\Rist_G^T(w)\leq\Rist_G^{\widehat T}(\widehat v),
	\]
	so
	\[
	\Rist_G^{\widehat T}(\widehat v)\neq 1.
	\]
	Thus \(G\) acts weakly branch on \(\widehat T\).

	We now compare the branch properties of the two actions. For a rooted
	tree \(S\), write
	\[
	\Rist_G^S(n)
	=
	\prod_{v\in S_n}\Rist_G^S(v)
	\]
	for the rigid stabiliser of level \(n\).
	
	Since the level-\(n\) partition associated to \(\widehat c\) refines
	the level-\(n\) partition associated to \(c\), every cylinder at level
	\(n\) of \(\widehat T\) is contained in a cylinder at level \(n\) of
	\(T\). Hence
	\[
	\Rist_G^{\widehat T}(n)
	\leq
	\Rist_G^T(n)
	\]
	for every \(n\). It follows that if the action on \(\widehat T\) is
	branch, then the action on \(T\) is branch.
	
	Conversely, fix \(n\geq 0\). Each level-\(n\) cylinder of
	\(\widehat T\) is clopen in \(\partial T\), hence both compact and open. It is therefore  a finite union
	of cylinders arising from \(T\), since these form a basis for the topology. Since there are only finitely many
	vertices at level \(n\) of \(\widehat T\), there is some \(m\) such that
	every level-\(m\) cylinder of \(T\) is contained in a level-\(n\)
	cylinder of \(\widehat T\). Therefore
	\[
	\Rist_G^T(m)
	\leq
	\Rist_G^{\widehat T}(n).
	\]
	If the action of \(G\) on \(T\) is branch, then
	\(\Rist_G^T(m)\) has finite index in \(G\), and hence so does
	\(\Rist_G^{\widehat T}(n)\). Thus the action on \(\widehat T\) is
	branch. We have proved that the action on \(\widehat T\) is branch if
	and only if the action on \(T\) is branch. This proves the third assertion. 
	
	Since \(\widehat c\) is invariant under the canonical action of
	\(\Aut(G)\) on \(\partial T\), \(\Aut(G)\) acts on \(\widehat T\) by rooted automorphisms, by Proposition~\ref{prop:reconstruction}. This action is faithful by Corollary~\ref{cor: injective to boundary}, since
	 its induced action on
	\(\partial\widehat T\cong\partial T\) is the faithful canonical
	boundary action.  More precisely, Proposition~\ref{prop:reconstruction} shows that \( \pi_{\widehat{c}}: \partial T \to \partial \widehat{T}\) is an \( \Aut(G)\)-equivariant homeomorphism. 
	
	The compatibility relation
	\[
	\alpha\cdot(g\xi)
	=
	\alpha(g)(\alpha\cdot\xi)
	\]
	shows that this action normalises \(G\), and that conjugation by the tree
	automorphism induced by \(\alpha\) restricts to the automorphism
	\(\alpha\) of \(G\). Hence conjugation gives a surjective homomorphism
	\[
	N_{\Aut(\widehat T)}(G)\longrightarrow\Aut(G).
	\]
	
	Its kernel is the centraliser of \(G\) in \(\Aut(\widehat T)\). Every
	element of this centraliser induces a homeomorphism of
	\(\partial\widehat T\cong\partial T\) which centralises \(G\). By
	Theorem~\ref{thm:boundary-rigidity}, the centraliser of \(G\) in
	\(\Homeo(\partial T)\) is trivial. The kernel is therefore trivial, and
	so
	\[
	N_{\Aut(\widehat T)}(G)\cong\Aut(G).
	\]
	
	This proves the fourth point. 
	
	Finally, for the final point, \(\Aut(G)\) contains \(G\) in its action on \(\widehat T\), via the identification of \( G\) with \( \Inn(G)\). The compatibility relation above immediately implies that the action of \( \Inn(G)\) is the same as that of \( G\), and this is also apparent from the normaliser equation.

	Since \(G\) acts transitively on every level of \(\widehat T\),
	so does \(\Aut(G)\). Moreover, for every vertex \(\widehat v\),
	\[
	1\neq\Rist_G^{\widehat T}(\widehat v) \cong \Rist_{\Inn(G)}^{\widehat T}(\widehat v)
	\leq
	\Rist_{\Aut(G)}^{\widehat T}(\widehat v).
	\]
	Thus \(\Aut(G)\) acts weakly branch on \(\widehat T\).
\end{proof}

\begin{remark}
The overlap function $\widehat c$ is characterised intrinsically as the greatest $\Aut(G)$-invariant overlap function dominated by $c$. Indeed, if $d$ is $\Aut(G)$-invariant and $d\leq c$, then
\[
  d\leq c^\alpha
\]
for every $\alpha$, and therefore $d\leq\widehat c$.
\end{remark}

We note that commensurability of branch groups has been studied in \cite{Garridowilson} and not all finite index subgroups need be (weakly) branch. For instance, if \( G\) is branch and \( |T_n| > 1\) then \( \Rist_G(n)\) is a finite index subgroup admitting normal subgroups which intersect trivially, which implies that  \( \Rist_G(n)\) is not branch by \cite[Theorem 1.1 and Lemma 2.2(a)]{Garridowilson}.

As an application of our main theorem, we study finite extensions of (weakly) branch groups and determine when they are (weakly) branch. 

\begin{theorem}\label{thm:finite-extensions}
	Let \(G\) be a finitely generated weakly branch group, and let
	\[
	1\longrightarrow G\longrightarrow\Gamma
	\overset{q}{\longrightarrow} F\longrightarrow 1
	\]
	be a short exact sequence, where \(F\) is finite. Let
	\[
	\mu\colon F\longrightarrow\Out(G)
	\]
	be the monodromy homomorphism induced by conjugation in \(\Gamma\), and let
	\[
	\pi\colon\Aut(G)\longrightarrow\Out(G)
	\]
	be the natural quotient map.
	
	Conjugation induces a surjective homomorphism
	\[
	\rho\colon\Gamma\longrightarrow
	\pi^{-1}\bigl(\mu(F)\bigr)\leq\Aut(G)
	\]
	whose kernel is naturally isomorphic to \(\ker(\mu)\).
	Consequently, the following are equivalent:
	\begin{enumerate}[(i)]
		\item The monodromy homomorphism \(\mu\) is injective.
		\item The homomorphism \(\rho\) is an isomorphism.
		\item The group \(\Gamma\) is weakly branch.
	\end{enumerate}
	
	If \(G\) is branch, these conditions are also equivalent to:
	\begin{enumerate}[(i)]
		\setcounter{enumi}{3}
		\item The group \(\Gamma\) is branch.
	\end{enumerate}
	
	Whenever these conditions hold, every faithful weakly branch action
	\(G\curvearrowright T\) gives rise to a weakly branch action
	\(\Gamma\curvearrowright\widehat T\), where \(\widehat T\) is the tree
	supplied by Theorem~\ref{thm:main}. In fact \(G\curvearrowright T\) is
	branch if and only if \(\Gamma\curvearrowright\widehat T\) is branch.
\end{theorem}

\begin{proof}
	Since \(G\) is weakly branch, it is centreless by Lemma~\ref{lem:centreless}. Conjugation therefore
	defines a homomorphism
	\[
	\rho\colon\Gamma\longrightarrow\Aut(G)
	\]
	whose restriction to \(G\) identifies \(G\) with \(\Inn(G)\). Moreover, \( \pi\bigl(\rho(\Gamma)\bigr)=\mu(F)\). 
	It follows that
	\[
	\rho(\Gamma)=\pi^{-1}\bigl(\mu(F)\bigr).
	\]
	
	Now let
	\(
	C=\ker(\rho)=C_\Gamma(G).
	\)
	Since
	\(
	C\cap G=Z(G)=1,
	\)
	the restriction of \(q\) to \(C\) is injective. Its image is precisely
	\(\ker(\mu)\). Indeed, every element of \(C\) induces the trivial outer
	automorphism of \(G\). Conversely, if \(\gamma G\in\ker(\mu)\), then
	conjugation by \(\gamma\) on \(G\) is conjugation by some \(g\in G\),
	and hence \(g^{-1}\gamma\in C\). Thus
	\[
	C\cong q(C) = \ker(\mu).
	\]
	In particular, \(\mu\) is injective if and only if \(\rho\) is an
	isomorphism. This proves the equivalence of (i) and (ii).
	
	Suppose that these conditions hold. Choose a faithful weakly branch
	action \(G\curvearrowright T\), and let \(\widehat T\) be supplied by
	Theorem~\ref{thm:main}. Via \(\rho\), the group \(\Gamma\) is a subgroup
	of \(\Aut(G)\) containing \(G\), and hence acts on \(\widehat T\).
	Since \(G\) is level-transitive, so is \(\Gamma\), and for every vertex
	\(v\) of \(\widehat T\),
	\[
	1\neq\Rist_G(v)\leq\Rist_\Gamma(v).
	\]
	Thus \(\Gamma\) acts weakly branch on \(\widehat T\), proving (iii).
	
	Conversely, suppose that \(\Gamma\) is weakly branch. The subgroup
	\[
	C\cong\ker(\mu)
	\]
	is finite and normal in \(\Gamma\). Since a weakly branch group has no
	non-trivial finite normal subgroups by Lemma~\ref{lem:no finite normal subs}, \(C=1\), and hence \(\mu\) is
	injective. This proves the equivalence of (i)--(iii).

Finally, suppose that the equivalent conditions hold. By
Theorem~\ref{thm:main}, the action \(G\curvearrowright T\) is branch if
and only if \(G\curvearrowright\widehat T\) is branch.

For every vertex \(v\) of \(\widehat T\),
\[
\Rist_G(v)=\Rist_\Gamma(v)\cap G,
\]
which has finite index in \(\Rist_\Gamma(v)\), since \(G\) has finite
index in \(\Gamma\). As every level is finite, it follows that
\(\Rist_G(n)\) has finite index in \(\Rist_\Gamma(n)\). This implies that
\(\Rist_G(n)\) has finite index in \(G\) if and only if
\(\Rist_\Gamma(n)\) has finite index in \(\Gamma\). Hence
\(G\curvearrowright T\) is branch if and only if
\(\Gamma\curvearrowright\widehat T\) is branch, proving the final
assertion and the equivalence with~(iv).

\end{proof}

\section{A branch action which is not automorphism-invariant}
\label{sec:example}
We give an example in which the automorphism-invariant refinement is
strict.

Let \(H\) be a finitely generated group acting faithfully and branch on
a locally finite rooted tree \(T\), and let
\[
c\colon\partial T\times\partial T\longrightarrow\mathbb N\cup\{\infty\}
\]
be the associated overlap function. Let
\[
X=C_2^2
\]
and choose a subgroup \(U<X\) of order \(2\). Set
\[
W=H\wr X=H^X\rtimes X,
\]
where \(X\) acts regularly on itself and hence permutes the factors of
\(H^X\). We write \(H_x\) for the copy of \(H\) indexed by \(x\in X\).

Consider the compact totally disconnected space (four copies of \( \partial T\))
\[
Y=X\times\partial T
\]
and define
\[
c_U\colon Y\times Y\longrightarrow\mathbb N\cup\{\infty\}
\]
by
\[
c_U\bigl((x,\xi),(y,\eta)\bigr)=
\begin{cases}
	0,              & x-y\notin U,\\
	1,              & x-y\in U\setminus\{0\},\\
	2+c(\xi,\eta),  & x=y.
\end{cases}
\]

The corresponding level equivalence relations are particularly simple.
At level \(1\),
\[
(x,\xi)\sim_1(y,\eta)
\quad\Longleftrightarrow\quad
x+U=y+U,
\]
while at level \(2\),
\[
(x,\xi)\sim_2(y,\eta)
\quad\Longleftrightarrow\quad
x=y.
\]
More generally, for \(n\geq0\),
\[
(x,\xi)\sim_{n+2}(y,\eta)
\quad\Longleftrightarrow\quad
x=y\ \text{and}\ c(\xi,\eta)\geq n.
\]
Thus \(c_U\) is an admissible overlap function. Its associated rooted
tree, denoted by \(T_U\), has the form
\[
\{*\}\longrightarrow X/U\longrightarrow X,
\]
with a copy of \(T\) attached below each \(x\in X\).

The wreath product \(W\) acts on \(Y\) in the usual way. The base group
acts coordinatewise:
\[
(h_z)_{z\in X}\cdot(x,\xi)=(x,h_x\xi),
\]
and the top group acts by translations:
\[
q\cdot(x,\xi)=(q+x,\xi).
\]
The function \(c_U\) is invariant under these actions, so \(W\) acts
faithfully on \(T_U\).

This action is level-transitive. Moreover,
\[
H^X\leq\Rist_W(1)
\]
and, for every \(n\geq0\),
\[
\Rist_H(n)^X\leq\Rist_W(n+2).
\]
Indeed, the product of the groups \(H_x\) over \(x\) in a fixed coset of
\(U\) is supported inside the corresponding first-level cylinder, while
the latter inclusions follow from the copies of the rigid stabilisers of
\(H\) inside the subtrees below the vertices \(x\in X\). Since \(H\) is
branch,
\[
[W:H^X]=|X|<\infty
\]
and
\[
[W:\Rist_H(n)^X]
=|X|\,[H:\Rist_H(n)]^{|X|}<\infty.
\]
Consequently, \(W\curvearrowright T_U\) is branch.

We now show that this action is not invariant under all automorphisms of
\(W\). Let \(V<X\) be another subgroup of order \(2\). Since
\[
\Aut(X)=\GL_2(\mathbb F_2)\cong S_3
\]
acts transitively on the subgroups of order \(2\), there exists
\(A\in\Aut(X)\) such that
\[
A(V)=U.
\]
The homeomorphism
\[
\Phi_A\colon Y\longrightarrow Y,
\qquad
\Phi_A(x,\xi)=(Ax,\xi),
\]
normalises \(W\): it sends \(q\in X\) to \(Aq\) and relabels the base
factor \(H_x\) as \(H_{Ax}\). Thus conjugation by \(\Phi_A\) induces an
automorphism
\[
\alpha_A\in\Aut(W).
\]

For this automorphism,
\[
\begin{aligned}
	(c_U)^{\alpha_A}\bigl((x,\xi),(y,\eta)\bigr)
	&=c_U\bigl((Ax,\xi),(Ay,\eta)\bigr)\\
	&=c_V\bigl((x,\xi),(y,\eta)\bigr),
\end{aligned}
\]
where the second equality follows from \(A(V)=U\). Thus every partition
of \(X\) into cosets of a subgroup of order \(2\) occurs as the
level-one partition of an automorphic twist of \(c_U\).

The three subgroups of order \(2\) in \(X=C_2^2\) determine the three
partitions of \(X\) into two pairs, whose common refinement is the
partition into singletons. Consequently, the automorphism-invariant
refinement \(\widehat T_U\) has at least four vertices at its first
level, whereas \(T_U\) has only two. In particular,
\[
\widehat c_U\neq c_U
\qquad\text{and}\qquad
\widehat T_U\neq T_U.
\]

\section{Finite branch extensions of the first Grigorchuk group}
\label{sec:grigorchuk-extensions}

 The result of \cite{GrigorchukSidki2004} realises the automorphisms of the first Grigorchuk group as automorphisms of the tree, so the discussion below is already implicit in the literature. In terms of our construction, \( T = \widehat{T}\) in this case. However, it is an illustration of what one can do and, in principle, applies to any branch group in the following way.
 
 Given a finitely generated branch group \( G\), its finite branch extensions are obtained from finite subgroups of \( \Out(G)\) which are then pulled back to \( \Aut(G)\). The isomorphism types of these finite extensions are given by conjugacy in the commensurator group of \( G\), as below. We use Grigorchuk's group to illustrate the method, as all the ingredients are already known.

 For further results on the automorphisms and automorphism tower of the
 first Grigorchuk group, see~\cite{BartholdiSidki}. Related abstract
 commensurability properties were established in
 \cite{GrigorchukWilsonCommensurability}.
 
 \medspace
 
 For now, we establish some elementary facts and notation.  Recall that, given a group \( G\), the commensurator group \( \Comm(G)\) is the group of all isomorphisms between finite index subgroups of \(G\), where two elements are considered equal if they agree on a finite index subgroup of \( G\).
 
 We then have, 
   
 \begin{lemma}
 	\label{lem:aut into comm}
 	Let \( G\) be a weakly branch group. Then the natural map \( \Aut(G)\) embeds into \( \Comm(G)\). Equivalently, the only automorphism of \( G\) which is the identity on a finite index subgroup of \(G\) is the identity map.
 \end{lemma}
 \begin{proof}
 	Let \( T\) be the locally finite tree on which \( G\) acts in a weakly branch way. By Theorem~\ref{thm:boundary-rigidity}, \( \Aut(G) \)  has a canonical faithful action on \( \partial T\), satisfying the compatibility relation above.

 	Now suppose that \( \alpha \in \Aut(G)\) acts as the identity on some finite index subgroup, \( H\). Further suppose that, for some \( \xi \neq \eta \in \partial T\), we have \( \alpha \xi = \eta\). Then, as \( \partial T \) is Hausdorff, there exists an open \( U \subseteq \partial T\) such that \( \alpha U \cap U = \emptyset\).
 	
 	Now \( U \) contains some cylinder \( \cyl(v)\), and we note that all the elements of \( \Rist_G(v) \) act on the boundary with support contained in \( \cyl(v)\). As \( G\) is weakly branch, \( \Rist_G(v)\) is infinite and hence \( H \cap \Rist_G(v)\) is non-trivial since \( H\) has finite index. 
 	
 	Therefore there is some \( 1 \neq h \in H\) whose support, \( \text{supp}_{\partial T}(h)\), on the boundary is contained in \( U\) and non-empty. 
 	
 	Now, compatibility and the fact that \( \alpha\) fixes \( h\) imply that
 	\[
 	\alpha \cdot (h \eta ) = h( \alpha \cdot \eta) \text{ for all }  \eta \in \partial T.
 	\]
 	In particular, 
 	\[
 	\alpha(U) \supseteq \alpha(\text{supp}_{\partial T}(h) ) = \text{supp}_{\partial T}(h) \subseteq U.
 	\]

 	Since \( \text{supp}_{\partial T}(h) \neq \emptyset \), this is a contradiction which shows that \( \alpha\) must act as the identity on the boundary and hence \( \alpha\) is the identity as the action is faithful.

 \end{proof}

 We then have
 
 \begin{proposition}
 	Let \( G\) be a centreless group in which \( \Aut(G)\) embeds into \( \Comm(G)\) under the natural map. Let \( E, F \leq \Aut(G)\) be subgroups of the automorphism group of \( G\), both of which contain \( \Inn(G)\) as a (normal) subgroup of finite index.

 	Then \( E\) and \( F\) are abstractly isomorphic if and only if they are conjugate as subgroups of \( \Comm(G)\). 
 \end{proposition}
 \begin{proof}
 	For this proof we will write conjugation using exponentiation notation: namely, $a^b:=b^{-1} a b$. We will also write \( \Phi\) on the right to match our conjugation convention. 
 	
 	The hypotheses give us that 
 	\[
 	G \cong \Inn(G) \leq \Aut(G) \leq \Comm(G),
 	\]
 	and we will argue by dealing with \( \Inn(G)\) directly. 
 	
 	We write \( \ad (g)\) to denote the inner automorphism induced by conjugation by \( g\), for \( g \in G\). Now note that if \( \phi \) is a commensuration and \( g\) is in the domain of \( \phi\), then \( \ad(g)^{\phi} = \ad( g \phi)\). Hence if \( \phi, \psi \in \Comm(G)\) and \( u^{\phi} = u^{\psi}\) for all \( u \in U\), where \( U\) is some finite index subgroup of \( \Inn(G)\), then \( \phi = \psi\) in \( \Comm(G)\). We shall use this below.

 	\medspace
 	
 	Onto the argument; since conjugation induces an isomorphism, one direction is clear. 
 	
 	For the other direction let us suppose that \( \Phi: E \to F\) is an isomorphism. Then there is a finite index subgroup \( U \) of \( \Inn(G)\) such that \( U \Phi \) is also a finite index subgroup of \( \Inn(G)\). Thus \( \Phi\) induces a commensuration of \( G\), which we call \( \phi\). But note that \( U\) is a group of inner automorphisms which embeds into \( \Comm(G)\). We therefore have that,
 	\[
 	u^{\phi} = u \Phi, \text{ for all } u \in U.
 	\] 	
 	

 	Then, for any \( e \in E\),
 	\[
 	(u^e)\Phi = (u \Phi)^{e\Phi} = u^{\phi \cdot (e \Phi)},
 	\]
 	for all \( u \in U\).

 	On the other hand, 
 	
 	\[
 	(u^e)\Phi = u^{e \cdot \phi},
 	\]
 	as long as \( u^e \in U\). 
 	
 	Hence, conjugation by  \( \phi \cdot (e\Phi) \) and \( e \cdot\phi \) agree on \( U \cap U^{e^{-1}}\), which has finite index in \( \Inn(G)\). Hence \( e^{\phi}  = e \Phi\) as elements of \( \Comm(G)\).

 	This shows that \( E\) and \(F\)
are conjugate by \( \phi\)  in \( \Comm(G)\)	
 \end{proof}

\medskip

Let $\mathfrak G$ denote the first Grigorchuk group. Since $\mathfrak G$ is
centreless, we identify $\mathfrak G$ with $\Inn(\mathfrak G)$ and write
\[
\pi\colon \Aut(\mathfrak G)\longrightarrow \Out(\mathfrak G)
\]
for the natural quotient map.

Grigorchuk and Sidki \cite[Theorem 1]{GrigorchukSidki2004} proved that
\[
\Out(\mathfrak G)\cong \bigoplus_{\mathbb N} C_2;
\]
in particular, $\Out(\mathfrak G)$ is an elementary abelian $2$-group of
countably infinite rank. For every $d\geq 0$, choose a
$d$-dimensional subgroup
\[
F_d\leq \Out(\mathfrak G)
\]
and set
\[
E_d=\pi^{-1}(F_d).
\]
By Theorem~\ref{thm:finite-extensions}, the group $E_d$ is branch and fits into
an exact sequence
\[
1\longrightarrow \mathfrak G
\longrightarrow E_d
\longrightarrow F_d
\longrightarrow 1.
\]
In particular,
\[
[E_d:\mathfrak G]=|F_d|=2^d.
\]

We show that the groups $E_d$ are pairwise non-isomorphic. For this we recall
the relative modular homomorphism of an abstract commensurator.

Let $G$ be a group. An element of $\Comm(G)$ is represented by an isomorphism
\[
\theta\colon U\longrightarrow V
\]
between finite-index subgroups $U,V\leq G$. Define
\[
\Delta_G(\theta)
=\frac{[G:V]}{[G:U]}
\in \mathbb Q_{>0}^{\times}.
\]
This is independent of the chosen representative of the commensuration and
defines a homomorphism
\[
\Delta_G\colon\Comm(G)\longrightarrow\mathbb Q_{>0}^{\times}.
\]
Indeed, restricting $\theta$ to a finite-index subgroup changes the numerator
and denominator by the same factor, while multiplicativity follows by
restricting two composable commensurations to suitable finite-index
subgroups. 


We shall use the following elementary observation.

\begin{lemma}
	\label{lem:extension-index-modular}
	Let $G$ be a group and let $E$ and $E'$ contain copies of $G$ as subgroups of
	finite index. If
	\[
	\Phi\colon E\longrightarrow E'
	\]
	is an isomorphism, then $\Phi$ induces a commensuration $\theta$ of $G$
	satisfying
	\[
	\Delta_G(\theta)
	=\frac{[E:G]}{[E':G]}.
	\]
\end{lemma}

\begin{proof}
	Put
	\[
	U=G\cap\Phi^{-1}(G)
	\qquad\text{and}\qquad
	V=G\cap\Phi(G).
	\]
	Then $\Phi$ restricts to an isomorphism
	\[
	\theta=\Phi|_U\colon U\longrightarrow V,
	\]
	and hence determines an element of $\Comm(G)$.
	
	Write
	\[
	m=[E:G]
	\qquad\text{and}\qquad
	n=[E':G].
	\]
	Since $\Phi(G)$ has index $m$ in $E'$, counting the index of $V$ in $E'$ in
	two ways gives
	\[
	n[G:V]
	=[E':V]
	=m[\Phi(G):V].
	\]
	Moreover, $\Phi$ induces an isomorphism from $G$ to $\Phi(G)$ taking $U$ to
	$V$, and therefore
	\[
	[\Phi(G):V]=[G:U].
	\]
	Consequently,
	\[
	\frac{[G:V]}{[G:U]}=\frac{m}{n},
	\]
	as required.
\end{proof}

Röver proved that the abstract commensurator $\Comm(\mathfrak G)$ is a
non-abelian simple group \cite[Theorem 1.3]{Rover}. It follows that every homomorphism from
$\Comm(\mathfrak G)$ to an abelian group is trivial. In particular,
\[
\Delta_{\mathfrak G}=1.
\]

Suppose now that $E_d\cong E_e$. By
Lemma~\ref{lem:extension-index-modular}, the corresponding commensuration
$\theta\in\Comm(\mathfrak G)$ satisfies
\[
\Delta_{\mathfrak G}(\theta)
=\frac{[E_d:\mathfrak G]}{[E_e:\mathfrak G]}
=2^{d-e}.
\]
Since $\Delta_{\mathfrak G}$ is trivial, we obtain $2^{d-e}=1$, and hence
$d=e$. We have therefore proved the following.

\begin{corollary}
	\label{cor:infinitely-many-grigorchuk-extensions}
	The first Grigorchuk group has infinitely many pairwise non-isomorphic finite
	branch extensions.
\end{corollary}

\begin{remark}
	The preceding argument distinguishes extensions arising from subgroups of
	$\Out(\mathfrak G)$ of different dimensions. It does not address whether two
	distinct $d$-dimensional subgroups of $\Out(\mathfrak G)$ can give
	non-isomorphic extensions for a fixed value of $d$.
\end{remark}

\bibliographystyle{plain}

\bibliography{refs}

\end{document}